\documentclass{amsart}

\usepackage{amsmath,amssymb,amsfonts,enumerate,amsthm,graphicx,color}

\def\opn#1#2{\def#1{\operatorname{#2}}} 
\opn\chara{char} \opn\length{\ell} \opn\pd{pd} \opn\rk{rk}
\opn\projdim{proj\,dim} \opn\injdim{inj\,dim} \opn\rank{rank}
\opn\depth{depth} \opn\grade{grade} \opn\height{height}
\opn\embdim{emb\,dim} \opn\codim{codim}

\opn\Tr{Tr} \opn\bigrank{big\,rank}
\opn\superheight{superheight}\opn\lcm{lcm}
\opn\trdeg{tr\,deg}%
\opn\reg{reg} \opn\lreg{lreg} \opn\skel{skel}
\opn\multideg{multideg}
\opn\div{div} \opn\Div{Div} \opn\cl{cl} \opn\Cl{Cl}
\opn\Spec{Spec} \opn\Supp{Supp} \opn\supp{supp} \opn\Sing{Sing}
\opn\Ass{Ass}
\opn\Ann{Ann} \opn\Rad{Rad} \opn\Soc{Soc}
\opn\Ker{Ker} \opn\Coker{Coker} \opn\Im{Im} \opn\Hom{Hom}
\opn\Tor{Tor} \opn\Ext{Ext} \opn\End{End} \opn\Aut{Aut}
\opn\id{id}

\opn\nat{nat}
\opn\pff{pf}
\opn\Pf{Pf} \opn\GL{GL} \opn\SL{SL} \opn\mod{mod} \opn\ord{ord}
\opn\aff{aff} \opn\con{conv} \opn\relint{relint} \opn\st{st}
\opn\lk{lk} \opn\cn{cn} \opn\core{core} \opn\vol{vol}
\opn\link{link} \opn\star{star} \opn\skel{skel} \opn\Reg{Reg}
\opn\gr{gr}

\def\pot#1#2{#1[\kern-0.28ex[#2]\kern-0.28ex]}

\opn\dirlim{\underrightarrow{\lim}}
\opn\inivlim{\underleftarrow{\lim}}
\def\Implies{\ifmmode\Longrightarrow \else
     \unskip${}\Longrightarrow{}$\ignorespaces\fi}
\def\implies{\ifmmode\Rightarrow \else
     \unskip${}\Rightarrow{}$\ignorespaces\fi}
\def\iff{\ifmmode\Longleftrightarrow \else
     \unskip${}\Longleftrightarrow{}$\ignorespaces\fi}

\let\:=\colon

\newtheorem{thm}{Theorem}[section]
\newtheorem{cor}[thm]{Corollary}

\newtheorem{defn}[thm]{Definition}

\newtheorem{exam}[thm]{Example}
\newtheorem{rem}[thm]{Remark}

\numberwithin{equation}{section}

\begin{document}
\bibliographystyle{amsplain}

\title{Cohen-Macaulay and Gorenstein path ideals of trees }
\author{Sara Saeedi Madani and Dariush Kiani }
\thanks{2010 \textit{Mathematics Subject Classification.} 13F55, 05C75, 05C05.}
\thanks{\textit{Key words and phrases.} Path ideals, Cohen-Macaulay, Gorenstein, Matroid, Fitting t-partitioned tree. }

\address{Sara Saeedi Madani, Department of Pure Mathematics,
 Faculty of Mathematics and Computer Science,
 Amirkabir University of Technology (Tehran Polytechnic),
424, Hafez Ave., Tehran 15914, Iran, and School of Mathematics, Institute for Research in Fundamental Sciences (IPM),
P.O. Box 19395-5746, Tehran, Iran.} \email{sarasaeedi@aut.ac.ir}
\address{Dariush Kiani, Department of Pure Mathematics,
 Faculty of Mathematics and Computer Science,
 Amirkabir University of Technology (Tehran Polytechnic),
424, Hafez Ave., Tehran 15914, Iran, and School of Mathematics, Institute for Research in Fundamental Sciences (IPM),
P.O. Box 19395-5746, Tehran, Iran.} \email{dkiani@aut.ac.ir}

\begin{abstract}

\noindent Let $R=k[x_{1},\ldots,x_{n}]$, where $k$ is a field. The
path ideal (of length $t\geq 2$) of a directed graph $G$ is the
monomial ideal, denoted by $I_{t}(G)$, whose generators
correspond to the directed paths of length $t$ in $G$. Let
$\Gamma$ be a directed rooted tree. We characterize all such trees whose path ideals are unmixed and Cohen-Macaulay.
 Moreover, we show that $R/I_{t}(\Gamma)$ is Gorenstein if and only if
the Stanley-Reisner simplicial complex of $I_{t}(\Gamma)$ is a matroid.

\end{abstract}

\maketitle

\section{Introduction}

\noindent Let $G$ be a directed graph over $n$ vertices and $t$ be a fixed integer such that $2\leq t\leq n$.
A sequence $v_{i_{1}},\ldots ,v_{i_{t}}$ of
distinct vertices, is called a \textbf{path} of length $t$ if there are
$t-1$ distinct directed edges $e_{1},\ldots ,e_{t-1}$ where
$e_{j}$ is a directed edge from $v_{i_{j}}$ to $v_{i_{j+1}}$.
Then the \textbf{path ideal} of $G$ of length $t$ is the monomial ideal
$$I_{t}(G)=(x_{i_{1}}\cdots x_{i_{t}} : v_{i_{1}},\ldots ,v_{i_{t}}~\mathrm{is~a~path~of~length}~t~\mathrm{in}~G)$$
in the polynomial ring $R=k[x_{1},\ldots ,x_{n}]$ over a
field $k$. Some properties of the path ideal of cycles and trees were studied in \cite{VJ} and \cite{SKT}.

In this paper, we focus on the path ideals of trees. Throughout the paper, we mean by tree, a directed rooted tree and by a path, a directed path.
We investigate some algebraic properties of this ideal. In \cite{V}, a characterization of all trees whose edge ideals, that is the case $t=2$, are Cohen-Macaulay is given. Also, in \cite{HHZ}, a characterization of all chordal graphs whose edge ideals are Cohen-Macaulay (resp. Gorenstein) is given. When $G$ is a tree, it is obviously chordal. So, their results also hold for trees. In this paper, for a tree, we generalize these results for all $t\geq 2$.

This paper is organized as follows. In the next section, we recall several definitions and terminology which we need later.
In Section 3, we characterize all trees whose path ideals are unmixed and hence Cohen-Macaulay. For this purpose, we use the correspondence between clutters and simplicial complexes and the fact that the facet simplicial complex associated to the paths of length $t$ of $\Gamma$ is a simplicial tree. In Section 4, we show that complete intersection and Gorenstein properties of the path ideal of a tree are equivalent to the property that the tree has only one directed path of length $t$. Moreover, we prove that this is the case if and only if the Stanley-Reisner simplicial complex of the path ideal is a matroid. Finally, we deduce that these conditions are equivalent to Cohen-Macaulayness of all symbolic and ordinary powers of the ideal.

\section{Preliminaries}

\noindent A \textbf{simplicial complex} $\Delta$ on the vertex set
$V(\Delta)=\{v_{1},\ldots ,v_{n}\}$ is a collection of subsets of $V=V(\Delta)$
such that if $F\in \Delta$ and $G\subseteq F$, then $G\in  \Delta$. (We sometimes write $[n]$ for the set of vertices of a simplicial complex or a graph).\\
An element in $\Delta$ is called a \textbf{face} of $\Delta$, and
$F\in \Delta$ is said to be
a \textbf{facet} if $F$ is maximal with respect to inclusion. Let $F_{1},\ldots,F_{q}$ be all
the facets of simplicial complex
$\Delta$. We sometimes write $\Delta=\langle F_{1},\ldots,F_{q}\rangle$.\\
A vertex which is contained only in one facet, is called a \textbf{free} vertex of $\Delta$.\\
For every face $G\in \Delta$, we define the \textbf{star} and \textbf{link} of $G$ as below $$\mathrm{st}_{\Delta}G=\{F\in \Delta~:~G\cup F\in \Delta\},$$
$$\mathrm{lk}_{\Delta}G=\{F\in \Delta~:~G\cap F=\emptyset~,~G\cup F\in \Delta\}.$$
 The \textbf{dimension} of a face $F$ is $|F|-1$. Let
$d=\textrm{max}\{|F|~:~F\in \Delta \}$, then the
\textbf{dimension} of $\Delta$, denoted by $\textrm{dim}(\Delta)$,
is $d-1$. We say that $\Delta$ is \textbf{pure} if all its facets have the same dimension.\\
Let $f_{i}=f_{i}(\Delta)$ denote the number of faces of dimension
$i$. The sequence $f(\Delta)=(f_{0},f_{1},\ldots,f_{d-1})$ is
called the $f$-vector of $\Delta$. By the convention, we set $f_{-1}=1$.\\
The \textbf{reduced Euler characteristic} $\widetilde{\chi}(\Delta)$ of $\Delta$ is given by $$\widetilde{\chi}(\Delta)=-1+\sum_{i=0}^{d-1}(-1)^if_{i}.$$
The \textbf{facet ideal} of $\Delta$ is $$I(\Delta)=(\prod_{x\in F}x : F\mathrm{~is~a~facet~of~}\Delta).$$\\
Now we define the simplicial complex $\Delta_{t}(G)$ to be
$$\Delta_{t}(G)= \langle \{v_{i_{1}},\ldots ,v_{i_{t}}\} :
v_{i_{1}},\ldots ,v_{i_{t}}~\mathrm{is~a~path~of~length~t~in}~G\rangle,$$
where $G$ is a directed graph. So we have $I_{t}(G)=I(\Delta_{t}(G))$.\\
The \textbf{Stanley-Reisner ideal} of $\Delta$ is the monomial
ideal
$$I_{\Delta}=(\prod_{x\in F}x : F\notin \Delta).$$
The \textbf{Stanley-Reisner ring} of $\Delta$ is $k[\Delta]=R/I_{\Delta}$.\\
Let $\Delta = \langle F_{1},\ldots ,F_{q} \rangle$. A
\textbf{vertex cover} of $\Delta$ is a subset $A$ of $V$, with
the property that for every facet $F_{i}$ there is a vertex
$v_{j}\in A$ such that $v_{j}\in F_{i}$. A \textbf{minimal vertex
cover} of $\Delta$ is a subset $A$ of $V$ such that $A$ is a
vertex cover
and no proper subset of $A$ is a vertex cover of $\Delta$. The minimum number of vertices in a vertex cover is called the
\textbf{covering number} of $\Delta$, and it coincides with the height of $I(\Delta)$, ht$(I(\Delta))$.
A simplicial complex $\Delta$ is \textbf{unmixed} if all of its minimal vertex covers have the same cardinality.

Recall that a finitely generated graded module $M$ over a
Noetherian graded $k$-algebra $S$ is said to satisfy the Serre's
condition $S_{r}$ if depth $M_{P}\geq $min$(r,$ dim $M_{P})$, for
all $P\in \mathrm{Spec} (S)$. Thus, $M$ is Cohen-Macaulay if and only if it satisfies the
Serre's condition $S_r$ for all $r$.

A graded $R$-module $M$ is called \textbf{sequentially
Cohen-Macaulay} (resp. \textbf{sequentially $S_r$}) (over $k$) if there exists a finite filtration of graded $R$-modules
$0=M_{0}\subset M_{1}\subset\cdots \subset M_{r}=M$
such that each $M_{i}/M_{i-1}$ is Cohen-Macaulay (resp. $S_r$), and the Krull dimensions of the quotients
are increasing, i.e.
$$ \mathrm {dim} (M_{1}/M_{0}) < \mathrm{dim}(M_{2}/M_{1}) < \cdots < \mathrm{dim}(M_{r}/M_{r-1}).$$
\begin{thm}\label{seq-unmixed}
{\em (}see \cite [Lemma 3.6]{FV} and \cite[Corollary 2.7]{HTYZ}{\em)} Let $I$ be a squarefree monomial ideal in
$R=k[x_{1},\ldots,x_{n}]$. Then $R/I$ is Cohen-Macaulay {\em (}resp. $S_r${\em)} if and
only if $R/I$ is sequentially Cohen-Macaulay {\em (}resp. sequentially $S_r${\em)} and $I$ is unmixed.
\end{thm}

A \textbf{clutter} $\mathcal{C}$ with finite vertex set $X$ is a family of subsets of $X$, called edges,
none of which is included in another. The set of vertices and edges of $\mathcal{C}$ are denoted
by $V(\mathcal{C})$ and $E(\mathcal{C})$, respectively. The
set of edges of a clutter can be viewed as the set of facets of a simplicial complex.\\
Let $\mathcal{C}$ be a clutter with finite vertex set $X=\{v_1,\ldots ,v_n\}$ with no isolated vertices, i.e., each vertex occurs in at least one
edge. The edge ideal
of $\mathcal{C}$, denoted by $I(\mathcal{C})$, is the ideal of $R$ generated by all monomials
$\prod_{v_i\in E} x_i$ such that $E\in E(\mathcal{C})$. The edge ideal of a clutter could be seen as the facet ideal of its corresponding
simplicial complex.\\
A clutter has the \textbf{K\"{o}nig property} if
the maximum number of pairwise disjoint edges equals the covering number. A
\textbf{perfect matching of $\mathcal{C}$ of K\"{o}nig type} is a collection $E_1,\ldots ,E_g$ of pairwise disjoint
edges whose union is $X$ and such that $g$ is the height of $I(\mathcal{C})$.\\
Let $A$ be the incidence matrix of a clutter $\mathcal{C}$. A clutter $\mathcal{C}$ has a
cycle of length $r$ if there is a square sub-matrix of $A$ of order $r\geq 3$ with exactly
two 1's in each row and column. A clutter without odd cycles is called \textbf{balanced}
and an acyclic clutter is called \textbf{totally balanced}.

A \textbf{leaf} of a simplicial complex $\Delta$ is a facet $F$ of $\Delta$ such that either $F$ is the only facet of $\Delta$, or there exists a facet $G$ in
$\Delta$, $G\neq F$, such that $F\cap F' \subseteqF\cap G$ for every facet $F'\in \Delta$, $F'\neq F$. A simplicial complex $\Delta$ is a called \textbf{simplicial tree} if
$\Delta$ is connected and every non-empty subcomplex $\Delta'$ contains a leaf. By a subcomplex,
we mean any simplicial complex of the form $\Delta'=\langle F_{i_{1}},\ldots ,F_{i_{q}} \rangle$, where $\{F_{i_{1}},\ldots ,F_{i_{q}}\}$ is a subset of the facets of $\Delta$.
We adopt the convention that the empty simplicial complex is also a simplicial tree. A simplicial complex $\Delta$ with the property that every connected component of $\Delta$ is a simplicial tree is called a \textbf{simplicial forest}.\\
In \cite{HHTZ}, it was shown that a clutter $\mathcal{C}$ is totally balanced if and only if it is the clutter of the facets of a simplicial forest
\cite[Theorem 3.2]{HHTZ}.\\
Moreover, in \cite{F}, it was shown that a simplicial tree (forest) has the K\"{o}nig property \cite[Theorem 5.3]{F}.

\section{ Trees with Cohen-Macaulay path ideals }

\noindent A tree $\Gamma$ can be viewed as a directed graph by picking any vertex of $\Gamma$ to
be the root of the tree, and assigning to each edge the direction ``away" from the
root. Because $\Gamma$ is a tree, the assignment of a direction will always be possible. A
leaf is any vertex in $\Gamma$ adjacent to only one other vertex. The level of a vertex $v$,
denoted level$(v)$, is one fewer than the length of the unique path starting at the root and ending at $v$. The height of a tree, denoted height$(\Gamma)$, is then given by
height$(\Gamma):=\mathrm{max}_{v\in V}$level$(v)$.

\begin{exam}\label{a}
\em{Let $\Gamma$ be the tree in the Figure \ref{1}, in which $v_1$ is the root and $\mathrm{height}(\Gamma)=3$. Also, let $t=4$. Then we have $$I_4(\Gamma)=(x_1x_2x_4x_8,x_1x_2x_4x_9,x_1x_3x_6x_{10},x_1x_3x_7x_{11}).$$}
\begin{center}
\begin{figure}
\hspace{0 cm}
\includegraphics[height=3.8cm,width=4.9cm]{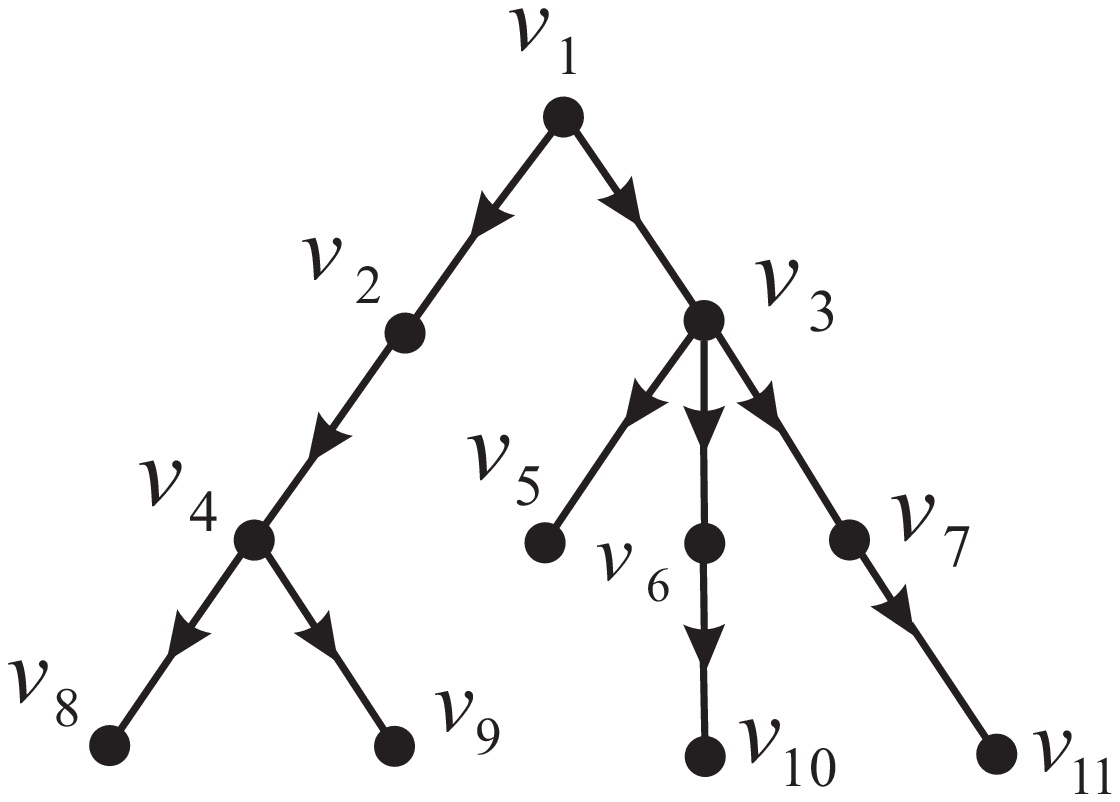}\\
\caption{\footnotesize{}}\hspace{4 cm}
\label{1}
\end{figure}
\end{center}
\end{exam}
Throughout the paper, we mean by a tree, a directed rooted tree as above and by a path, a directed path. By abuse of notation, we use $F=\{v_{i_{1}},\ldots,v_{i_{t}}\}$ where $\mathrm{level}(v_{i_1})<\cdots <\mathrm{level}(v_{i_t})$, to denote the path of length $t$ in a tree $\Gamma$ which starts from $v_{i_{1}}$ and ends at $v_{i_{t}}$, and also the corresponding facet in $\Delta_t({\Gamma})$.

In \cite{VJ}, it was shown that:

\begin{thm}\label{simplicial tree}
\cite[Theorem 2.7]{VJ} Let $\Gamma$ be a tree over $n$ vertices and $2 \leq t\leq n$. Then $\Delta_t(\Gamma)$ is a simplicial tree.
\end{thm}

\begin{thm}\label{seq. C-M}
\cite[Corollary 2.12]{VJ} Let $\Gamma$ be a tree over $n$ vertices and $2\leq t\leq n$. Then $R/I_t(\Gamma)$ is sequentially Cohen-Macaulay.
\end{thm}

In this section, we focus on some other properties of the path ideal of a tree. We determine when this ideal is unmixed and hence Cohen-Macaulay.

\begin{rem}\label{C(Gamma)}
\em{Note that by removing leaves at level strictly less than $(t-1)$ from a tree $ \Gamma$ and repeating this process until $\Gamma$ has no more such leaves, one obtains a tree denoted by $C(\Gamma)$.
In \cite{BHO}, this process is called \textbf{cleaning process} and the tree $C(\Gamma)$ is called the \textbf{clean form} of $\Gamma$. Note that the generators of $I_t(\Gamma)$ and $I_t(C(\Gamma))$ are the same (but in different polynomial rings). So, the graded Betti numbers of these two ideals are also the same.}
\end{rem}

Now, we want to introduce a class of trees which plays an important role in the main result of this section.

\begin{defn}\label{partitioned}
\em{Let $\Gamma$ be a tree over $n$ vertices and $2\leq t\leq n$. Suppose that $F_1,\ldots ,F_m$ are all facets of $\Delta=\Delta_t(C(\Gamma))$ containing a leaf of $C(\Gamma)$ such that each leaf belongs to exactly one of them. If $V(\Delta)$ is the disjoint union of $F_1,\ldots ,F_m$, then we say that $\Gamma$ is \textbf{t-partitioned} (by $F_1,\ldots ,F_m$).

Now, let $\Gamma$ be a t-partitioned tree (by $F_1,\ldots ,F_m$). We define a \textbf{t-branch} of $\Gamma$, as a path of length $t+1$, say $P$, which starts at a vertex of some $F_i$, like $x$, and $P\cap F_i=\{x\}$. Then, for each $i=1,\ldots,m$, we define \textbf{degree} of $F_i$, as
$$\mathrm{Deg}_{\Gamma}(F_i):=\mathrm{the~number~of~vertices~of}~F_i~\mathrm{which~are~the~first~vertices~of~a~}t-\mathrm{branch~of}~\Gamma.$$
Moreover, we define degree of $\Gamma$, as $$\mathrm{Deg}(\Gamma):=\mathrm{max}\{\mathrm{Deg}_{\Gamma}(F_i)~:~1\leq i\leq m\}.$$
We call a t-branch of $\Gamma$, \textbf{initial} if it intersects some $F_i$ in the first vertex of $F_i$. Otherwise, we call it \textbf{non-initial}.

Also, we define \textbf{level} of a t-branch $P$ of $\Gamma$, denoted by level$(P)$, as the level of the vertex $x$, where $P\cap F_i=\{x\}$ for some $i=1,\ldots,m$.
}
\end{defn}

\begin{defn}\label{fitting partitioned}
\em{Let $\Gamma$ be a t-partitioned tree over $n$ vertices and $2\leq t\leq n$. We say that $\Gamma$ is \textbf{fitting t-partitioned}, if the following hold:
\begin{item}
\item (1) $\mathrm{Deg}(\Gamma)\leq 1$; and
\item (2) $\mathrm{level}(P)\leq t-1$, for each non-initial t-branch $P$ of $\Gamma$.
\end{item}
}
\end{defn}

\begin{exam}\label{b}
\em{(a) Let $\Gamma$ be the tree in Figure 1 and $t=4$. Then, the set of vertices of $C(\Gamma)$ is not disjoint union of $F_1=\{v_1,v_2,v_4,v_8\}$, $F_2=\{v_1,v_2,v_4,v_9\}$, $F_3=\{v_1,v_3,v_6,v_{10}\}$ and $F_4=\{v_1,v_3,v_7,v_{11}\}$. So, $\Gamma$ is not 4-partitioned. Note that by cleaning $\Gamma$,
the only vertex which is removed, is $v_5$.

(b) Let $\Gamma_1$ be the tree in Figure 2  and $t=3$. Note that vertex $v_3$ is removed in $C(\Gamma_1)$. So, the vertex set of $C(\Gamma_1)$ is the disjoint union of $F_1=\{v_1,v_4,v_7\}$, $F_2=\{v_2,v_5,v_8\}$ and $F_3=\{v_6,v_9,v_{10}\}$ and hence  $\Gamma_1$ is 3-partitioned (by $F_1,F_2,F_3$). The 3-branches of $\Gamma_1$ are $P_1=\{v_1,v_2,v_5,v_8\}$, $P_2=\{v_1,v_2,v_6,v_9\}$ and $P_3=\{v_2,v_6,v_9,v_{10}\}$ which are all initial. Also, $P_1$ and $P_2$ intersect $F_1$, and $P_3$ intersects $F_2$. We have level$(P_1)=\mathrm{level}(P_2)=\mathrm{level}(v_1)=0$ and level$(P_2)=\mathrm{level}(v_2)=1$. Moreover, note that Deg$_{\Gamma_1}(F_1)=\mathrm{Deg}_{\Gamma_1}(F_2)=1$, Deg$_{\Gamma_1}(F_3)=0$ and hence Deg$(\Gamma_1)=1$. Thus, by Definition \ref{fitting partitioned}, $\Gamma_1$ is a fitting 3-partitioned tree.
\begin{center}
\begin{figure}
\hspace{0 cm}
\includegraphics[height=4.3cm,width=4.2cm]{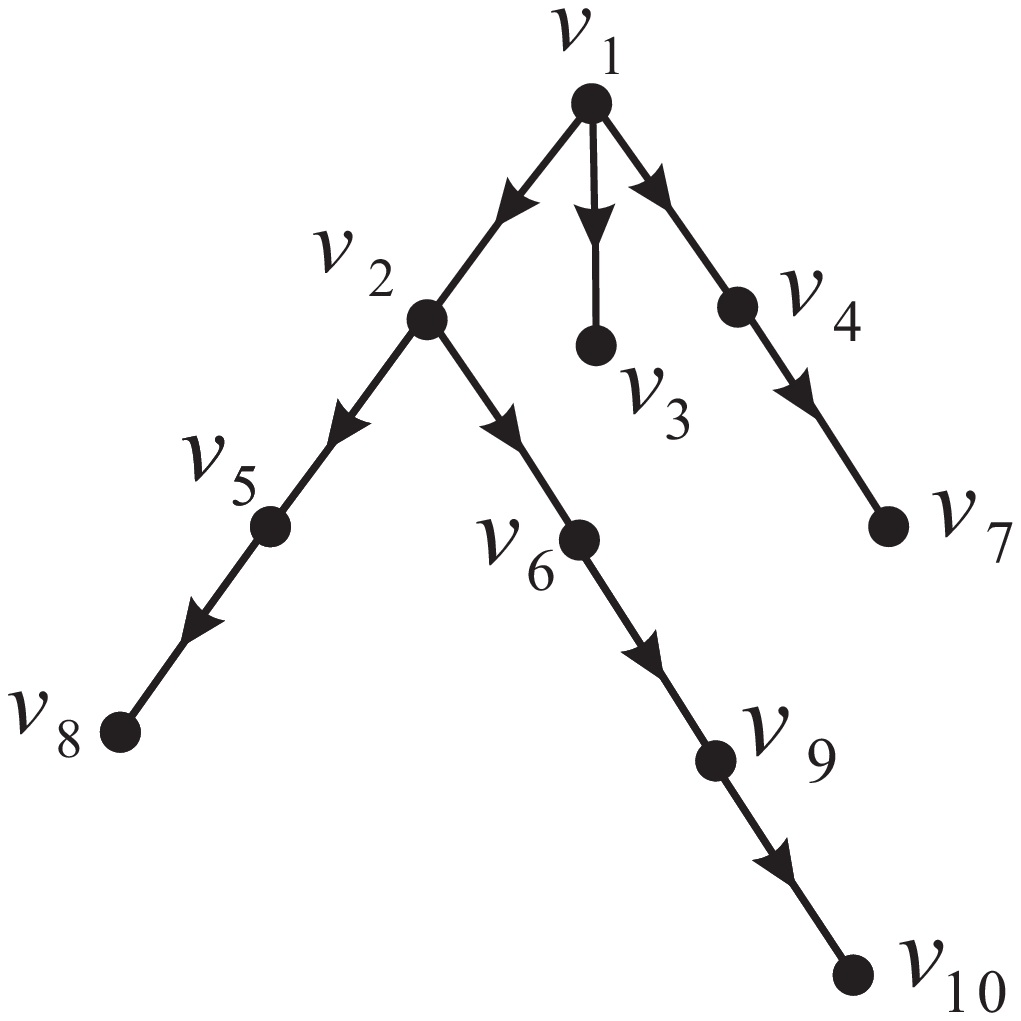}
\caption{\footnotesize{}}\hspace{4 cm}
\label{1}
\end{figure}
\end{center}
(c)  Let $\Gamma_2$ be the tree in Figure 3 and $t=3$. We have $C(\Gamma_2)=\Gamma_2$. Also, $F_1=\{v_1,v_2,v_3\}$, $F_2=\{v_4,v_5,v_7\}$ and $F_3=\{v_6,v_8,v_9\}$ are the facets mentioned in Definition \ref{partitioned}. The vertex set of $\Gamma_2$ is the disjoint union of $F_1$, $F_2$ and $F_3$. So, $\Gamma_2$ is 3-partitioned. In addition, $P_1=\{v_2,v_4,v_5,v_7\}$, $P_2=\{v_2,v_4,v_5,v_6\}$, and $P_3=\{v_5,v_6,v_8,v_9\}$ are the only 3-branches of $\Gamma_2$, where both of them are non-initial and we have level$(P_1)=\mathrm{level}(P_2)=\mathrm{level}(v_2)=1$ and level$(P_3)=\mathrm{level}(v_5)=3$. Although Deg$(\Gamma)=$Deg$_{\Gamma_2}(F_1)=\mathrm{Deg}_{\Gamma_1}(F_2)=1$, $\Gamma_2$ is not fitting 3-partitioned, since level$(P_2)=3>2$.
\begin{center}
\begin{figure}
\hspace{0 cm}
\includegraphics[height=4.9cm,width=2.8cm]{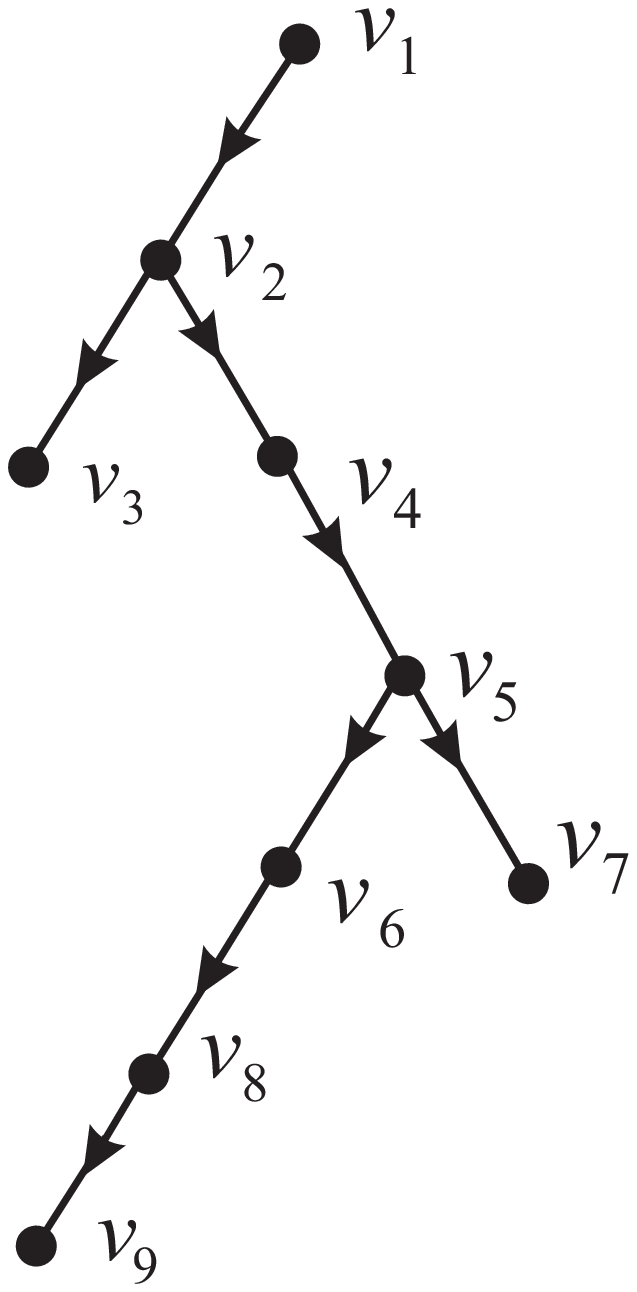}
\caption{\footnotesize{}}\hspace{4 cm}
\label{1}
\end{figure}
\end{center}
(d) Let $\Gamma_3$ be the tree in Figure 4 and $t=3$. We have $C(\Gamma_3)=\Gamma_3$. Also, the vertex set of $\Gamma_3$ is the disjoint union of $F_1=\{v_2,v_5,v_8\}$, $F_2=\{v_1,v_3,v_6\}$ and $F_3=\{v_4,v_7,v_9\}$. So, $\Gamma_3$ is 3-partitioned. In addition, $P_1=\{v_1,v_2,v_5,v_8\}$ and $P_2=\{v_1,v_4,v_7,v_9\}$ are the only 3-branches of $\Gamma_3$, where both of them are initial and we have level$(P_1)=\mathrm{level}(P_2)=\mathrm{level}(v_1)=0$. Also, we have Deg$(\Gamma_3)=$Deg$_{\Gamma_3}(F_2)=1$ and so $\Gamma_3$ is a fitting 3-partitioned tree.
\begin{center}
\begin{figure}
\hspace{0 cm}
\includegraphics[height=3.8cm,width=4.9cm]{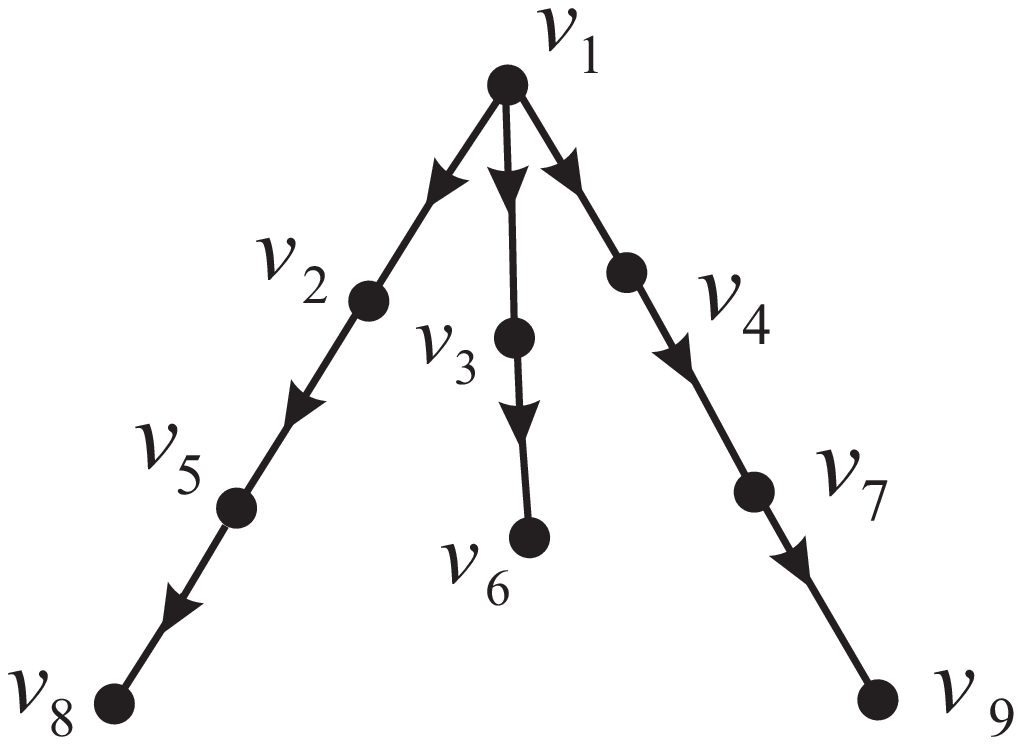}
\caption{\footnotesize{}}\hspace{4 cm}
\label{1}
\end{figure}
\end{center}
}
\end{exam}

\begin{rem}
\em{Note that in Definition \ref{partitioned}, the only case in which a leaf might belong to more than one of the facets $F_1,\ldots,F_m$, is when the root of the tree is also a leaf. For instance, you can see in the Example \ref{b}, part (c), (see Figure 3), that we do not consider $\{v_1,v_2,v_4\}$ as some $F_i$, since the root, $v_1$, also belongs to $F_1=\{v_1,v_2,v_3\}$ and considering $F_1$ is necessary, as the leaf $v_3$ just belongs to it. }
\end{rem}

We need the following theorem to prove the main result of this section:

\begin{thm}\label{clutter}
\cite[Corollary 2.19]{MRV} Let $\mathcal{C}$ be a totally balanced clutter with the K\"{o}nig property. Then $\mathcal{C}$ is unmixed
if and only if there is a perfect matching $E_1,\ldots,E_g$ of K\"{o}nig type such that $E_i$ has a free vertex for all $i$, and for
any two edges $E,E'$ of $\mathcal{C}$ and for any edge $E_i$ of the perfect matching, one has that $E\cap E_i \subset E'\cap E_i$ or $E'\cap E_i \subset E\cap E_i$.
\end{thm}

The next theorem is the main theorem of this section:

\begin{thm}\label{unmixed}
Let $\Gamma$ be a tree over $n$ vertices and $2\leq t\leq n$. Then $I_t(\Gamma)$ is unmixed if and only if $\Gamma$ is fitting t-partitioned.
\end{thm}

\begin{proof}
By Theorem \ref{simplicial tree} and Remark \ref{C(Gamma)}, we have $\Delta=\Delta_t(C(\Gamma))$ is a simplicial tree. Moreover, a simplicial tree is totally balanced by \cite[Theorem 3.2]{HHTZ} and also has the K\"{o}nig property by \cite[Theorem 5.3]{F}. Also, note that $\Delta=\Delta_t(\Gamma)$.\\

\textbf{``Only if"} Suppose that $I_t(\Gamma)$ is unmixed. So, $\Delta$ is unmixed. Thus, by Theorem \ref{clutter}, there exist disjoint facets $E_1,\ldots,E_g$ of $\Delta$ such that $E_i$ has a free vertex for all $i=1,\ldots,g$ and $V(\Delta)=\bigcup_{i=1}^{g}E_i$, where $g=\mathrm{ht}(I_t(\Gamma))$. Suppose that $F_1,\ldots ,F_m$ are all facets of $\Delta$ containing a leaf of $C(\Gamma)$ such that each leaf belongs to exactly one of them. First we show that $V(\Delta)=\bigcup_{j=1}^{m}F_j$. Let $E_i=\{v_{i_{1}},\ldots,v_{i_{t}}\}$ for each $i=1,\ldots,g$, where $\mathrm{level}(v_{i_1})<\cdots <\mathrm{level}(v_{i_t})$. Now fix an integer $i=1,\ldots,g$. We consider two following cases:

Case (1). Suppose that $v_{i_t}$ is a leaf of $C(\Gamma)$.  Note that if a leaf of $C(\Gamma)$ is not the root, then it is contained in exactly one facet of $\Delta$, that is some $F_j$. So, in this case, there exists $j_i\in \{1,\ldots,m\}$ such that $E_i=F_{j_i}$.

Case (2). Suppose that $v_{i_t}$ is not a leaf of $C(\Gamma)$. Then there exists a vertex $x$ with level greater than $v_{i_t}$'s and adjacent to $v_{i_t}$. Thus, $v_{i_{2}},\ldots,v_{i_{t}}$ are contained in the facet $G=\{v_{i_{2}},\ldots,v_{i_{t}},x\}$ and hence are not free. But $E_i$ has a free vertex, so $v_{i_1}$ should be free. Therefore, $v_{i_1}$ is the root of $C(\Gamma)$, since otherwise there is a vertex of level less than $v_{i_1}$ and adjacent to it and so $v_{i_1}$ is contained in another facet, a contradiction. If $v_{i_1}$ is not a leaf, then there exists a vertex $y\neq v_{i_2}$ adjacent to $v_{i_1}$ such that $\mathrm{level}(v_{i_2})=\mathrm{level}(y)$. Thus, $v_{i_1}$ and $y$ are contained in a path of length $t$ and so a facet of $\Delta$, since $C(\Gamma)$ does not have any leaves at level strictly less than $(t-1)$. So, $v_{i_1}$ is not a free vertex, a contradiction. Therefore, $v_{i_1}$ is a leaf of $C(\Gamma)$ which is just contained in $E_i$. So, by the way of picking $F_i$'s, there exists $j_i=1,\ldots,m$ such that $E_i=F_{j_i}$.

Thus, by the above cases and the fact $V(\Delta)=\bigcup_{i=1}^{g}E_i$, we have $V(\Delta)=\bigcup_{j=1}^{m}F_j$.

Now, we show that the $F_i$'s are disjoint. If for each $i=1,\ldots,m$, the last vertex of $F_i$ is a leaf, then $F_i$ is equal to some $E_{i_j}$ and so the result follows. If there exists some $F_i$ which contains the root, say $z$, as a leaf and its last vertex is not a leaf, then $z$ is only contained in $F_i$ and the other $F_i$'s are as the previous case and so they are disjoint, by our assumption. We may assume that $i=1$. Let $\alpha$ be the number of vertices of $F_1$ not contained in $\bigcup_{j=2}^{m}F_j$. So, $0< \alpha \leq t$, since $z$ has this property. On the other hand, $V(\Delta)$ is the disjoint union of $E_i$'s. Hence, we have $gt=(m-1)t+\alpha$. Thus, $\alpha=t$ and $g=m$. So, $F_i$'s are disjoint and so
$V(\Delta)$ is the disjoint union of $F_1,\ldots,F_m$. Hence, $\Gamma$ is t-partitioned (by $F_1,\ldots ,F_m$). Also, without loss of generality, we can assume that $F_i=E_i$ for each $i=1,\ldots,m$.

Now suppose that $\Gamma$ is not fitting t-partitioned. So, we have Deg$(\Gamma)>1$ or there exists a non-initial t-branch $P$ of $\Gamma$ such that level$(P)\geq t$.\\
If  Deg$(\Gamma)>1$, then there exists an integer $i=1,\ldots,m$ such that Deg$_{\Gamma}(F_i)>1$. Thus, $F_i$ contains at least two distinct vertices $v_{i_{s}}$ and $v_{i_{l}}$ which are the first vertices of two different t-branches, say $P$ and $P'$. So, $P$ and $P'$ are paths of length $t+1$ starting from $v_{i_{s}}$ and $v_{i_{l}}$, respectively. Clearly, by omitting the last vertex of $P$ (resp. $P'$), we get a path of length $t$ starting from $v_{i_{s}}$ (resp. $v_{i_{l}}$), say $P_{v_{i_{s}}}$ (resp. $P_{v_{i_{l}}}$). So, we have $P_{v_{i_{s}}}\cap F_i=\{v_{i_{s}}\}$ and $P_{v_{i_{l}}}\cap F_i=\{v_{i_{l}}\}$, none of them contains the other. By Theorem \ref{clutter}, it is a contradiction, since $\Delta$ is unmixed.\\
Now, suppose that there exists a non-initial t-branch $P$ of $\Gamma$ such that level$(P)\geq t$. Also, suppose that $v_{i_{s}}$ is the intersection of $P$ and some $F_i$.
So, $v_{i_{s}}$ is not the first vertex of $F_i$ and level$(v_{i_{s}})\geq t$. Let $P_{v_{i_{s}}}$ be a path of length $t$ starting from $v_{i_{s}}$ (as we discussed in the previous case). So, we have $P_{v_{i_{s}}}\cap F_i=\{v_{i_{s}}\}$. On the other hand, since level$(v_{i_{s}})\geq t$ and $v_{i_{s}}$ is not the first vertex of $F_i$, there is a path of length $t$ in $C(\Gamma)$ ending at $v_{i_{s-1}}$, say $H$. Thus $H\cap F_i=\{v_{i_{1}},\ldots,v_{i_{s-1}}\}$. Therefore,
none of $H\cap F_i$ and $P_{v_{i_{s}}}\cap F_i$ contains the other, again a contradiction, by Theorem \ref{clutter}. Thus $\Gamma$ is a fitting t-partitioned tree.\\

\textbf{``If"} Suppose that $\Gamma$ is a fitting t-partitioned tree (by $F_1,\ldots ,F_m$). We should show that $\Delta$ is unmixed. Since $F_i$'s are disjoint, we have $m\leq\mathrm{ht}(I_t(\Gamma))=$ covering number of $\Delta$. Let $v_{i_{1}}$ be the first vertex of $F_i$, for all $i=1,\ldots,m$. It is not difficult to see that $S=\{v_{1_{1}},\ldots,v_{m_{1}}\}$ is a minimal vertex cover of $\Delta$. Thus ht$(I_t(\Gamma))=m$. So, we have $F_1,\ldots,F_m$ is a perfect matching of K\"{o}nig type for $\Delta$, since $\Gamma$ is t-partitioned. Moreover, each $F_i$ contains a leaf of $C(\Gamma)$ and hence it has a free vertex. Therefore, by Theorem \ref{clutter}, it is enough to show that for any two facets $E$ and $E'$ of $\Delta$ and for each $F_i$, one has $E\cap F_i \subset E'\cap F_i$ or $E'\cap F_i \subset E\cap F_i$. So, fix an integer $i=1,\ldots,m$ and suppose that $E$ and $E'$ are two facets of $\Delta$. If $E\cap F_i=\emptyset$ or $E'\cap F_i=\emptyset$, then there is nothing to prove. So, suppose that both of the intersections are non-empty. Now, since $\mathrm{Deg}(\Gamma)\leq 1$, we can consider the following cases:

Case (1). Suppose that $\mathrm{Deg}_{\Gamma}(F_i)=0$. So, there does not exist any t-branch intersecting $F_i$. Thus, none of the
vertices of $F_i$ is contained in
some path of length $t$ whose last vertex does not belong to $F_i$, since the $F_i$'s are disjoint.
So, the only possible choice for $E$ and $E'$ is such that the last vertices of $E$ and $E'$ belong to $F_i$. Let $v_{i_{j}}$ and $v_{i_{l}}$ be the last vertices of $E$ and $E'$, respectively. Also, suppose that level$(v_{i_{j}})\leq$ level$(v_{i_{l}})$. Note that because $C(\Gamma)$ is a tree, there exists a unique path from the root to each vertex. So, we have
$E\cap F_i=\{v_{i_{1}},\ldots,v_{i_{j}}\}\subseteq \{v_{i_{1}},\ldots,v_{i_{l}}\}=E'\cap F_i$, as desired.

Case (2). Suppose that $\mathrm{Deg}_{\Gamma}(F_i)=1$. So, there is exactly one vertex $x$ in $F_i$ intersecting some t-branches of $\Gamma$. Thus, we can only choose those paths whose last vertices belong to $F_i$ or paths of the form $P_{x}$ (similar to what we explained in ``Only if" part) or paths whose last vertices belong to a path of the form $P_{x}$, as $E$ and $E'$. Note that, those paths whose last vertices belong to $F_i$ contains $x$, since $x$ is the first vertex of $F_i$ or level$(x)\leq t-1$. Thus, in each choice, we have $E\cap F_i\subseteq E'\cap F_i$ or $E'\cap F_i\subseteq E\cap F_i$. Therefore, similar to the previous case, we get the result.
\end{proof}

Combining Theorem \ref{unmixed}, Theorem \ref{seq. C-M} and Theorem \ref{seq-unmixed}, we have the following important corollary:

\begin{cor}\label{CM}
Let $\Gamma$ be a tree over $n$ vertices, $2 \leq t\leq n$ and $r\geq 2$. Then the following
conditions are equivalent:\\
\indent {\em {(i)}} $R/I_t(\Gamma)$ is unmixed.\\
\indent {\em {(ii)}} $R/I_t(\Gamma)$ is Cohen-Macaulay.\\
\indent {\em {(iii)}} $R/I_t(\Gamma)$ is $S_r$.\\
\indent {\em {(iv)}} $\Gamma$ is fitting t-partitioned.
\end{cor}

By the above corollary and Example \ref{b}, we have that $\Gamma$ and $\Gamma_2$ are not Cohen-Macaulay, but $\Gamma_1$ and $\Gamma_3$ are.

As a consequence of Corollary \ref{CM}, we have the following corollary on the simplest kind of trees, i.e. lines. By $L_n$, we mean the line over $n$ vertices with directed edges $e_1,\ldots,e_{n-1}$, where $e_i$ is from $v_i$ to $v_{i+1}$ for $i=1,\ldots,n-1$.

\begin{cor}\label{L_n}
Let $2\leq t\leq n$. Then $R/I_{t}(L_{n})$ is Cohen-Macaulay if and only if $t=n$ or $n/2$.
\end{cor}

\begin{rem}\label{pd}
\em{ Suppose that $F_1,\ldots ,F_m$ are all facets of $\Delta=\Delta_t(C(\Gamma))$ containing a leaf of $C(\Gamma)$ such that each leaf belongs to exactly one of them. Note that by the proof of Theorem \ref{unmixed}, if $R/I_t(\Gamma)$ is Cohen-Macaulay, then we have $\mathrm{ht}(I_t(\Gamma))=m$. So, $\mathrm{depth}(R/I_t(\Gamma))=\mathrm{dim}(R/I_t(\Gamma))=n-m$ and hence $\mathrm{pd}(R/I_t(\Gamma))=m$, by Auslander-Buchsbaum formula.}
\end{rem}

\begin{rem}\label{t=2}
\em{Note that for $t=2$, Corollary \ref{CM} yields the previous result on the Cohen-Macaulayness of the edge ideal of a tree (see \cite[Theorem 6.3.4]{V} and the main theorem of \cite{HHZ}). In the case $t=2$, there are not any differences between various directions assigning to $\Gamma$. So, one can pick each vertex as a root and obtain $I_2(\Gamma)=I(\Gamma)$.}
\end{rem}

\section{Trees with Gorenstein path ideals}

\noindent In this section, we determine complete intersection and Gorenstein path ideals of trees. Also, as a consequence, we we characterize those trees such that all powers of their path ideals are Cohen-Macaulay.

First recall that a \textbf{matroid} is
a collection of subsets of a finite set, called independent sets, with the following
properties:\\
\indent (i) The empty set is independent.\\
\indent (ii) Every subset of an independent set is independent.\\
\indent (iii) If $F$ and $G$ are two independent sets and $F$ has more elements than $G$, then
\indent there exists an element in $F$ which is not in $G$ that when added to $G$ still gives
\indent an independent set.\\
Clearly, we may consider a matroid as a simplicial complex.

Also, note that the path ideal of length $t$ of a tree $\Gamma$, can be viewed as a
Stanely-Reisner ideal of a simplicial complex $\Delta_{n,t}$ by setting: $F \subseteq [n]$ is a face of $\Delta_{n,t}$ if and only if $F$ contains no $t$ consecutive vertices. So, we have $I_{t}(\Gamma)=I_{\Delta_{n,t}}$.

Moreover, we need the following characterization of Gorenstein simplicial complexes:

\begin{thm}\label{Stanley}
\cite[Chapter II, Theorem 5.1]{S} Fix a field $k$ (or $\mathbb{Z}$). Let $\Delta$ be a simplicial complex and
$\Lambda:=\mathrm{core}(\Delta)$. Then the following are equivalent:\\
\indent {\em {(i)}} $\Delta$ is Gorenstein.\\
\indent {\em {(ii)}} either (1) $\Delta=\emptyset$, o , or o o , or (2) $\Delta$ is Cohen-Macaulay over $k$ of dimension
\indent $d-1\geq 1$, and the link of every $(d-3)$-face is either a circle or o-o or o-o-o , and
\indent $\widetilde{\chi}(\Lambda)={(-1)}^{\mathrm{dim(\Lambda)}}$ (the last condition
is superfluous over $\mathbb{Z}$ or if $\mathrm{char}(k)=2$).\\
Here, $\mathrm{core}(\Delta)=\Delta_{\mathrm{core}(V)}$, in which $\mathrm{core}(V)=\{v\in V~:~\mathrm{st}_{\Delta}\{v\}\neq \Delta\}$.
\end{thm}

Now we are ready to prove the main theorem of this section.

\begin{thm}\label{Gor}
Let $\Gamma$ be a tree over $n$ vertices and $2 \leq t\leq n$. Then the following
conditions are equivalent:\\
\indent {\em {(i)}} $R/I_t(\Gamma)$ is a complete intersection.\\
\indent {\em {(ii)}} $R/I_t(\Gamma)$ is Gorenstein.\\
\indent {\em {(iii)}} $\Delta_{n,t}$ is a matroid.\\
\indent {\em {(iv)}} $C(\Gamma)$ is $L_t$.
\end{thm}

\begin{proof}

(i) $\Rightarrow$ (ii) is clear.

(i) $\Rightarrow$ (iii) follows by \cite[Theorem 3.6 and Theorem 4.3]{TT}.

(ii) $\Rightarrow$ (iv) Suppose that $R/I_t(\Gamma)$ is Gorenstein. So, it is also Cohen-Macaulay and hence by Corollary \ref{CM}, $\Gamma$ is fitting t-partitioned (by $F_1,\ldots,F_m$). Without loss of generality, we assume that $F_1$ is the path containing the root of $\Gamma$. Moreover, let $F_i=\{v_{i_1},\ldots,v_{i_t}\}$, with $\mathrm{level}(v_{i_1})<\cdots <\mathrm{level}(v_{i_t})$, for all $i=1,\ldots,m$. Now, suppose on the contrary that $C(\Gamma)$ is not $L_t$. So, $m>1$, because $\Gamma$ is t-partitioned. Therefore, there exists some $F_i$ which constructs a path of length $t+1$ with a vertex $v_{1_s}$ of $F_1$, in which $s=1,\ldots,t$. In other words, $\{v_{1_s}\}\cup F_i$ is a t-branch of $\Gamma$. We assume that $i=2$. Let $G:=[n]\setminus \bigcup_{i=1}^{m}\{v_{i_1}\}$ for all $i=1,\ldots,m$. So, $G$ does not contain any $t$ consecutive vertices. Note that by Remark \ref{pd}, we have dim$(\Delta_{n,t})+1=$dim$(R/I_t(\Gamma))=n-m>1$. Now we consider two cases:
\\
Case (1). Let $s=1$. Then set $H:=G\setminus \{v_{1_t},v_{2_t}\}$. Note that $H$ does not contain any $t$ consecutive vertices. Hence it is a face of $\Delta_{n,t}$ of cardinality $n-m-2$. Also, we have $\mathrm{lk}_{\Delta_{n,t}} H=\langle \{v_{1_1},v_{2_t}\}, \{v_{1_t},v_{2_t}\},\{v_{1_t},v_{2_1}\}\rangle$, which is a path over four vertices.
\\
Case (2). Let $s>1$. Then set $H:=G\setminus \{v_{1_s},v_{2_t}\}$. Note that $H$ does not contain any $t$ consecutive vertices. Hence it is a face of $\Delta_{n,t}$ of cardinality $n-m-2$. Also, we have $\mathrm{lk}_{\Delta_{n,t}} H=\langle \{v_{1_1},v_{2_1}\}, \{v_{1_1},v_{2_t}\},\{v_{1_s},v_{2_t}\}\rangle$, which is a path over four vertices.
\\
Thus, by the above cases, we see that $\mathrm{lk}_{\Delta_{n,t}} H$ is not of the forms mentioned in Theorem \ref{Stanley}. So, ${\Delta_{n,t}}$ is not Gorenstein, a contradiction to the fact that $R/I_t(\Gamma)$ is Gorenstein.

(iii) $\Rightarrow$ (iv) Suppose that $\Delta_{n,t}$ is a matroid. For $t=2$, we have $C(\Gamma)=\Gamma$. So, if $\Gamma$ has more than one edge, then obviously $\Delta_{n,2}$, which is precisely the independence complex of $\Gamma$, is not a matroid, a contradiction. So, suppose that $t>2$. Note that every matroid is Cohen-Macaulay (see \cite[Theorem 3.4]{S}). So, we consider $F_1,\ldots,F_m$ similar to the previous part of the proof and suppose on the contrary that $m>1$. We assume that $F_1$, $F_2$ and $v_{1_{s}}$ are the same as in the previous part. Now we consider two cases: \\
Case (1). Let $s=1$. Then set $G:=(F_1\setminus \{v_{1_t}\})\cup (F_2\setminus \{v_{2_{(t-1)}},v_{2_t}\})$ and $H:=(F_1\setminus \{v_{1_1}\})\cup (F_2\setminus \{v_{2_t}\})$. Note that $G$ and $H$ do not contain any $t$ consecutive vertices. Hence they are faces of $\Delta_{n,t}$ of cardinality $2t-3$ and $2t-2$, respectively. On the other hand, $H\setminus G=\{v_{1_t},v_{2_{(t-1)}}\}$. But, $G\cup \{v_{1_t}\}$ and $G\cup \{v_{2_{(t-1)}}\}$ do not belong to $\Delta_{n,t}$, since both of them contain $t$ consecutive vertices. Thus, by definition, $\Delta_{n,t}$ is not a matroid, a contradiction.
\\
Case (2). Let $s>1$. Then set $G:=(F_1\setminus \{v_{1_{(s-1)}}\})\cup (F_2\setminus \{v_{2_{(t-1)}},v_{2_t}\})$ and $H:=(F_1\setminus \{v_{1_s}\})\cup (F_2\setminus \{v_{2_t}\})$. Note that $G$ and $H$ do not contain any $t$ consecutive vertices. Hence they are faces of $\Delta_{n,t}$ of cardinalities $2t-3$ and $2t-2$, respectively. On the other hand, $H\setminus G=\{v_{1_{(s-1)}},v_{2_{(t-1)}}\}$. But, we have $G\cup \{v_{1_{(s-1)}}\}$ and $G\cup \{v_{2_{(t-1)}}\}$ do not belong to $\Delta_{n,t}$, since both of them contain some $t$ consecutive vertices. Thus, by definition, $\Delta_{n,t}$ is not a matroid, a contradiction.\\
So, by the above cases, we get the desired result.

(iv) $\Rightarrow$ (i) is clear.
\end{proof}

\begin{rem}
\em{Notice that Theorem \ref{Gor} implies the result of \cite[Corollary 2.1]{HHZ} about Gorenstein property in the case that $t=2$ and $G$ is a tree.}
\end{rem}

Denote by $I^{(m)}$, the $m$-th symbolic power of the ideal $I$. We end this section with the following corollary which is obtained by Theorem \ref{Gor} and \cite[Theorem 3.6 and Theorem 4.3]{TT}:

\begin{cor}\label{powers}
Let $\Gamma$ be a tree over $n$ vertices, $2 \leq t\leq n$ and $I:=I_t(\Gamma)$. Then the following
conditions are equivalent:\\
\indent {\em {(i)}} $I^{m}$ {\em (}resp. $I^{(m)}${\em )} is Cohen-Macaulay for every $m\geq 1$.\\
\indent {\em {(ii)}} $I^{m}$ {\em (}resp. $I^{(m)}${\em )} is Cohen-Macaulay for some $m\geq 3$.\\
\indent {\em {(iii)}} $C(\Gamma)$ is $L_t$.
\end{cor}

\providecommand{\byame}{\leavevmode\hbox
to3em{\hrulefill}\thinspace}

\end{document}